\documentclass[10pt]{article}
\usepackage{amssymb}
\usepackage{indentfirst}
\usepackage{tikz}
\usepackage{amsmath,amsthm,amsfonts}
\newtheorem{Theorem}{Theorem}[section]
\newtheorem{Remark}[Theorem]{Remark}

\newtheorem{Lemma}[Theorem]{Lemma}

\newtheorem{Example}[Theorem]{Example}

\newcommand{\K}{\mathbb{K}}

\newcommand{\Z}{\mathbb{Z}}
\newcommand{\Q}{\mathbb{Q}}
\newcommand{\R}{\mathbb{R}}
\newcommand{\Zer}{\mathrm{Zer}}
\newcommand{\Res}{\mathrm{Res}}
\newcommand{\ord}{\mathrm{\mathrm{ord}}}
\newcommand{\trunc}{\mathrm{\mathrm{trunc}}}
\newcommand{\cont}{\mathrm{cont}}
\begin{document}

\title{On Abhyankar's irreducibility criterion for quasi-ordinary polynomials}
\footnotetext{
      \begin{minipage}[t]{4.2in}{\small
       }
       \end{minipage}}
 \author{Janusz Gwo\'zdziewicz and Beata Hejmej}
\maketitle

\begin{abstract}
Let  $f$ and $g$ be Weierstrass polynomials with coefficients 
in the ring of formal power series over an algebraically closed field of characteristic zero. 
Assume that  $f$ is irreducible and quasi-ordinary. 
We show that if degree of $g$ is small enough and all monomials appearing 
in the resultant of $f$ and~$g$ have orders big enough,
then $g$ is irreducible and quasi-ordinary, generalizing Abhyankar's irreducibility criterion 
for plane analytic curves. 
\end{abstract}

\section{Introduction}
\label{intro}
The paper is organized as follows. In the current section we introduce necessary notation 
and state our main result (Theorem~\ref{Th:irred1}). 
Its proof is in Section~3.  
Then, in Section~4, we introduce the notion of the logarithmic contact of irreducible Weierstrass polynomials and in 
Theorem~\ref{Th:irred3} rewrite the main results of the paper in terms of the logarithmic contact.
At the end of the paper we show that Abhyankar-Moh irreducibility criterion follows from Theorem~\ref{Th:irred1}.

Throughout the paper $\K$ is an algebraically closed  field of characteristic zero.
We use the notation $\K[[X]]$ for the ring $\K[[X_1,\dots,X_d]]$ of formal power series 
in $d$ variables with coefficients in $\K$ 
and the notation $\K[[X^{1/n}]]$ for the ring $\K[[X_1^{1/n},\dots,X_d^{1/n}]]$.
In one variable case the elements of this ring are called {\em Puiseux series}.
We will use a multi-index notation $X^q:=X_1^{q_1}\cdots X_d^{q_d}$ for  $q=(q_1,\dots,q_d)$. 

Let $f=Y^n+a_{n-1}(X)Y^{n-1}+\cdots+a_0(X)\in\K[[X]][Y]$ be a unitary polynomial.
Such a  polynomial is called \textit{quasi-ordinary} if its discriminant equals $u(X)X^q$ with $u(0)\neq0$. We call $f$ \textit{a Weierstrass polynomial} if $a_i(0)=0$ for all $i=0,\dots, n-1$. 

The classical Abhyankar-Jung theorem (see \cite{Parusinski-Rond})
states that every quasi-ordinary polynomial $f\in\K[[X]][Y]$ has its roots in $\K[[X^{1/m}]]$ 
for some positive integer $m$.
Hence one can factorize $f$ to the product $\prod_{i=1}^n(Y-\alpha_i)$, 
where $\alpha_i\in \K[[X^{1/m}]]$. We put $\Zer f=\{\alpha_1,\dots,\alpha_n\}$. 
Since the discriminant of a monic polynomial is a product of differences of its roots, 
we have $\alpha_i - \alpha_j=u_{ij}(X)X^{\lambda_{ij}}$ with $u_{ij}(0)\neq 0$. 
The $d$-tuple $d(\alpha_i,\alpha_j):=\lambda_{ij}$ of non-negative rational numbers 
will be called the \textit{contact between $\alpha_i$ and $\alpha_j$}.

For irreducible $f$ the contacts $d(\alpha,\alpha')$ for $\alpha,\alpha'\in\Zer f$, $\alpha\neq\alpha'$, 
are called the \textit{characteristic exponents} of~$f$. 

Let us introduce a partial order in the set $\Q_{\geq0}^d$:
$q\leq q'$ if and only if $q'-q\in \Q_{\geq0}^d$. 
Then the characteristic exponents
can be set to the increasing sequence $(h_1,\dots,h_s)$ (see \cite[Lemma~5.6]{Lipman}).
We call this sequence the \textit{characteristic} of $f$ and denote it by ${\rm Char}(f)$.

With the sequence of characteristic exponents we associate the increasing sequence of lattices 
$M_0\subset M_1\subset \dots\subset M_s$ defined as follows:
$M_0=\Z^d$ and
$M_i=\Z^d+\Z h_1+\cdots +\Z h_i$ for $i=1,\dots,s$.
We set $n_i=[M_i:M_{i-1}]$ for $i=1,\dots,s$, $n_{s+1}=1$ and 
$e_i=n_{i+1}\cdots n_{s+1}$ for $i=0,\dots,s$. 
Then $\deg f=n_1\cdots n_s$ (see \cite[Remark~2.7]{GP}). 
Finally we set 
\begin{equation}\label{Eq1}
q_{i}=\sum_{j=1}^{i} (e_{j-1}-e_j)h_j +e_{i} h_i
\end{equation}
for $i=1,\dots,s$. 

If $f,g\in \K[[X]][Y]$, then the resultant of this polynomials is denoted by $\Res(f,g)$.

\medskip
We can now formulate our main result.

\begin{Theorem}\label{Th:irred1}
Let $f\in \K[[X]][Y]$ be a quasi-ordinary irreducible polynomial of characteristic $(h_1,\dots,h_s)$ 
and let $g\in \K[[X]][Y]$ be a Weierstrass polynomial of degree $\leq n_1\cdots n_k$, 
where $1\leq k\leq s$.

If all monomials appearing in $\Res(f,g)$ have exponents greater than $(\deg g)q_k$ then
\begin{itemize}
\item[{\rm (i)}] $g$ is  irreducible and quasi-ordinary of degree $n_1\cdots n_k$ 
                        and characteristic $(h_1,\dots, h_k)$;
\item[{\rm (ii)}] for every $\gamma\in {\rm Zer}\, g$ there exists $\alpha\in {\rm Zer}\, f$ 
such that $\gamma-\alpha=\sum_{h>h_k}c_{h}X^{h}$.
\end{itemize} 

Moreover, if $X^{(\deg g)q_{k+1}}$ divides  $\Res(f,g)$ then
\begin{itemize}
\item[{\rm (iii)}] $\Res(f,g) = u(X)\, X^{(\deg g)q_{k+1}}$, where $u(0)\neq0$;
\item[{\rm (iv)}] for every $\gamma\in {\rm Zer}\, g$ there exists $\alpha \in {\rm Zer}\, f$ 
such that $\gamma-\alpha=c_{h_{k+1}}X^{h_{k+1}}+\sum_{h>h_{k+1}}c_{h}X^{h}$.
\end{itemize} 
\end{Theorem}

\begin{Remark}
{\rm In the point (iv) of Theorem~\ref{Th:irred1},
the monomial $X^{h_{k+1}}$ does not appear in the power series   $\gamma$.}
\end{Remark}
% Semi-roots z [G-P]

\begin{Example}
{\rm Let  
$f=Y^4-2X_1^3X_2^2Y^2-4X_1^5X_2^4Y-X_1^7X_2^6+X_1^6X_2^4$. 
The polynomial $f$ is quasi-ordinary and irreducible in $\mathbb{C}[[X_1,X_2]][Y]$  
with the  roots
\begin{eqnarray*}
\alpha_1&=&X_1^{3/2}X_2+X_1^{7/4}X_2^{3/2} \\
\alpha_2&=&X_1^{3/2}X_2-X_1^{7/4}X_2^{3/2} \\
\alpha_3&=&-X_1^{3/2}X_2+\sqrt{-1}X_1^{7/4}X_2^{3/2} \\
\alpha_4&=&-X_1^{3/2}X_2-\sqrt{-1}X_1^{7/4}X_2^{3/2}
\end{eqnarray*}
and  characteristic exponents $h_1=(\frac32,1)$ and $h_2=(\frac74,\frac32)$.

Let $g=(Y^2-X_1^3X_2^2)^2-4X_1^5X_2^4Y$.
Then $\Res(f,g)=X_1^{28}X_2^{24}$.
We have $X^{(\deg g)q_2}=X_1^{26}X_2^{20}$, 
so according to Theorem~\ref{Th:irred1}, 
the polynomial $g$ is irreducible and quasi-ordinary of characteristic $(h_1,h_2)$.}
\end{Example}
%%%%%%%%%%%%%%%%%%%%%%%%%%%%%%%%%%%

\section{Auxiliary results}
For $g=\sum_{a}c_a X^a \in\K[[X^{1/m}]]$
we define the {\it Newton polytope} $\Delta(g)$ as the convex hull 
of the set $\bigcup_{c_a\ne0}(a+\R_{\geq0}^d)$.
The Newton polytope $\Delta(f)$ 
of a polynomial $f\in\K[[X^{1/m}]][Y]$ is the Newton polytope of $f$ treated 
as an element of the ring $\K[[X_1^{1/m},\dots,X_d^{1/m},Y^{1/m}]]$.
In two variable case Newton polytopes are called {\em Newton polygons}.

Let $T$ be a single variable. The order of a fractional power series 
$\gamma\in \K[[T^{1/m}]]$ will be  denoted ${\rm ord}\, \gamma$. Note that for all $\alpha,\beta,\gamma\in \mathbb{K}[[T^{1/m}]]$ we have ${\rm ord}(\alpha-\beta)\geq \min\{{\rm ord}(\alpha-\gamma),{\rm ord}(\gamma-\beta)\}$. We call this property the {\it strong triangle inequality}.

\begin{Lemma}\label{L:comp}
Let $g$, $\tilde g\in\K[[T^{1/m}]] [Y]$ be Weierstrass polynomials 
such that $\Delta(g)=\Delta(\tilde g)$.  
Then $\{ \ord \gamma: \gamma \in \Zer g\}=\{ \ord \gamma: \gamma \in \Zer \tilde g\}$.
\end{Lemma}

\begin{proof}
The Newton polygon of the product $g=\prod_{i=1}^{\deg g}(Y-\gamma_i(T))$ 
is the Min\-kow\-ski sum of the Newton polygons of its factors and the shape of the 
Newton polygon of each factor $Y-\gamma_i(T)$ determines the order of $\gamma_i(T)$.

For a more detailed proof see  \cite[Theorem~2.1]{Ploski1}.
\end{proof}

Let  $\Q_+$ be the set of positive rational numbers. 
For a Newton polytope  $\Delta\subset \R_{\geq 0}^d$ 
and  $c\in\Q_+^d$  we define the {\em face} $\Delta^c:=\{v\in\Delta:\langle c,v\rangle=\min_{w\in\Delta}\langle c,w\rangle\}$.

We will say that a condition depending on $c\in\Q_{+}^d$  is satisfied for generic 
$c$ if it holds in an open and dense subset of $\Q_{+}^d$.

\begin{Lemma}\label{L:polygon}
Let $\Delta$ be the Newton polytope of some nonzero fractional power 
series~$\gamma\in \K[[X^{1/m}]]$.
Then for generic $c\in \Q_{+}^d$ a face $\Delta^c$ is a vertex of $\Delta$.
\end{Lemma}

\begin{proof}
Let $V$ be the (finite) set of vertices of $\Delta$. 
Then the set 
$$
U=\{c\in \Q_{+}^d: \forall v,w\in V (v\neq w \Rightarrow \langle c, v\rangle \neq \langle c, w\rangle) \}
$$
is open and dense in $\Q_{+}^d$ and for every $c\in U$ there is exactly one vertex 
$v$ of $\Delta$ such that $\langle c, v\rangle = \min \{\langle c, w\rangle: w\in V \}$. 
\end{proof}

With every $c=(c_1,\dots,c_d)\in \Q_{+}^d$ we associate the monomial substitution
$(X_1,\dots,X_d)=(T^{c_1},\dots,T^{c_d})$ written  $X=T^c$.
Applying this substitution to~$f=f(X,Y) \in \K[[X^{1/m}]][Y]$ we define 
$f^{[c]} := f(T^c,Y)\in \K[[T^{1/Nm}]][Y]$, 
where $N$ is a common denominator of coordinates of~$c$.

\begin{Lemma}\label{L:comp1}
Let $\gamma_1$, $\gamma_2 \in \K[[X^{1/m}]]$  be nonzero fractional power series. 
If  $\ord\gamma_1^{[c]} = \ord \gamma_2^{[c]}$ for generic $c\in\Q_{+}^d$,
then $\Delta(\gamma_1)=\Delta(\gamma_2)$. \label{cor:delta}
\end{Lemma}

\begin{proof}
Suppose that $\Delta(\gamma_1)\neq \Delta(\gamma_2)$.
Without loss of generality we may assume that  $\Delta(\gamma_1)\setminus \Delta(\gamma_2)$
is nonempty.  Since $\Delta(\gamma_2)$ is convex and closed, for any 
$v\in \Delta(\gamma_1)\setminus \Delta(\gamma_2)$ there exists
$c\in\R_{+}^d$ such that 
$\langle c,v\rangle < \inf_{w\in \Delta(\gamma_2)}\langle c,w\rangle$.
Then by Lemma~\ref{L:polygon} there is a vertex $v_0$ of $\Delta(\gamma_1)$
and an open set $U\subset \Q_+^d$ such that 
$\Delta(\gamma_1)^c=\{v_0\}$ and 
$\langle c,v_0\rangle < \inf_{w\in \Delta(\gamma_2)}\langle c,w\rangle$
for all $c\in U$. 
We get $\ord\gamma_1^{[c]}=\langle c, v_0 \rangle$ 
because in the fractional power series $\gamma_1^{[c]}$ 
there is no cancellation of the terms of order $\langle c, v_0 \rangle$ and 
$\ord\gamma_2^{[c]}>\langle c, v_0 \rangle$ (since all monomials appearing in~
$\gamma_2^{[c]}$ have orders bigger than $\langle c, v_0 \rangle$). 
Thus $\ord\gamma_1^{[c]} < \ord \gamma_2^{[c]}$ for $c\in U$. 
\end{proof}

\begin{Lemma}\label{Wn1}
Let $f\in\K[[X^{1/m}]][Y]$ be a nonzero polynomial. 
Given $c\in \Q_{+}^d$ we define the linear mapping 
$L_c:\R^d \times \R\to \R^2$, $L_{c}(x,y)=(\langle c, x \rangle,y)$.
Then for generic $c\in\Q_{+}^d$
$$\Delta( f^{[c]})=L_{c}(\Delta(f)).$$
\end{Lemma}

\begin{proof}
Write 
$f=a_n(X)Y^n+a_{n-1}(X)Y^{n-1}+\cdots+a_0(X)$ and
$ f^{[c]}=\bar a_n(T)Y^n+\bar a_{n-1}(T)Y^{n-1}+\cdots+\bar a_0(T)$.
By Lemma~\ref{L:polygon} for generic $c\in \Q_{+}^d$ and for any nonzero 
$a_i(X)$ the polygon  $\Delta(a_i(X))^{c}$ is a vertex of $\Delta(a_i(X))$.
Denote this vertex by $v_i$. 
Then $\ord \bar a_i(T)=\langle c, v_i \rangle$ because in the fractional power series 
$\bar a_i(T)=a_i(T^{c})$ 
there is no cancellation of the terms of the lowest order. 
Thus the vertices of $L_c(\Delta(f))$ belong to $\Delta( f^{[c]})$  which gives 
the desired equality.
\end{proof}

\begin{Remark}\label{Galois}
{\rm Let $K$ (respectively $L$) be the field of fractions of the ring $\K[[X]]$ 
(respectively $\K[[X^{1/m}]]$). Denote by ${\rm Gal}(L/K)$ the Galois group of the extension $K<L$. Then $L$ is normal over $K$ (as the splitting field of the family of polynomials $\{Y^m-X_i\in\mathbb{K}[[X]][Y]:i=1,\dots,d\}$) and every $\sigma\in {\rm Gal}(L/K)$ is given by 
$$
\sigma\Bigl( \sum_{a\in\mathbb{N}^d} c_a X^{a/m}\Bigr)= 
\sum_{a\in\mathbb{N}^d}\underline{\varepsilon}^{a}c_a X^{a/m}
$$
for some $\underline{\varepsilon} =(\varepsilon_1,\dots,\varepsilon_d)$ with $\varepsilon_l^m=1$. 
In particular, $\Delta(\sigma(\gamma))=\Delta(\gamma)$ for all nonzero $\gamma\in \K[[X^{1/m}]]$.}
\end{Remark}

For  $\alpha\in \K[[T^{1/m}]]$ and a finite set $A\subset \K[[T^{1/m}]]$ we define the {\it contact between} $\alpha$ and $A$ as ${\rm cont}(A,\alpha):=\max_{\gamma\in A} {\rm ord}(\alpha-\gamma)$.

From now on up to the end of this section we work under the assumption that  
$f\in \K[[X]][Y]$  is a quasi-ordinary irreducible polynomial of characteristic 
$(h_1,\dots,h_s)$ and ${\rm Zer}\, f=\{\alpha_1,\dots,\alpha_n\}$ is the set of its roots.

\begin{Theorem}\label{T21}
Let $g\in \K[[X]][Y]$ be a Weierstrass polynomial. 
If $c\in\Q_{+}^d$ is generic, then for any $\beta,\beta'\in\Zer f^{[c]}$ one has
 ${\rm cont}(\Zer\, g^{[c]},\beta)= {\rm cont}(\Zer\, g^{[c]},\beta')$. 
\end{Theorem}

\begin{proof}
For brevity we will write $\overline{p}$ instead of ${p}^{[c]}$ for every $p\in \K[[X^{1/m}]][Y]$. 

Since $f=\prod_{i=1}^n(Y-\alpha_i)$, we get $\bar f = \prod_{i=1}^n(Y-\bar\alpha_i)$ 
and consequently 
$\beta= \bar\alpha_i$, $\beta'=\bar\alpha_j$ for some $\alpha_i, \alpha_j\in\Zer f$. 
The roots $\alpha_i$, $\alpha_j$ are conjugate by the Galois automorphism. 
Hence by Remark~\ref{Galois} the Newton polytopes 
of $g_i=g(Y+\alpha_i)$ and $g_j=g(Y+\alpha_j)$ are equal.
By Lemma~\ref{Wn1}, 
the Newton polygons of $\bar g_i$ and  $\bar g_j$ are also equal. 

If $\Zer\, \bar g=\{\gamma_1,\dots,\gamma_k\}$ then 
$\Zer\, \bar g_i = \{\gamma_1-\beta,\dots,\gamma_k-\beta\}$ and 
$\Zer\, \bar g_j= \{\gamma_1-\beta' ,\dots,\gamma_k-\beta'\}$.
Hence it follows from Lemma~\ref{L:comp} that
${\rm cont}(\Zer\,\bar g,\beta)={\rm cont}(\Zer\,\bar g,\beta')$.
\end{proof}

If $A$ is any set, then $\# A$ denotes the cardinality of $A$.
%przeredagować 

\begin{Lemma}[Contact structure of $\Zer\,f$]\label{contact}
For every $\tilde\alpha \in\Zer f$ and $i\in\{1,\dots,s\}$ we have 
$\#\{\alpha \in {\rm Zer}\, f:d(\alpha,\tilde\alpha)> h_i\}=e_i$ and 
$\# \{\alpha\in {\rm Zer}\, f: d(\alpha,\tilde\alpha)=h_i\}=e_{i-1}-e_i.$
\end{Lemma}
\begin{proof}
See the proof of Proposition~3.1 from \cite{GP}.
\end{proof}

Fix $c\in\Q_{+}^d$ and for every $w\in\Q^d$ denote 
$\overline{w}:=\langle c,w\rangle$. 
We set $\overline{h_0}=0$, $\overline{h_{s+1}}=+\infty$, $\overline{q_0}=0$
and define a continuous function
$\phi_c:[0,+\infty)\to[0,+\infty)$ such that 
\begin{itemize}
\item[{\rm (i)}] $\phi_c(\overline{h_i})=\overline{q_i}$ for $i=0,\dots, s$;
\item[{\rm (ii)}] $\phi_c$ is linear in each interval $(\overline{h_i},\overline{h_{i+1}})$
                         for $i=0,\dots s$;
\item[{\rm (iii)}] the graph of $\phi_c$ has slope 1 over the interval $(\overline{h_s},+\infty)$.
\end{itemize}

\begin{Lemma}\label{L:ineq}
The function $\phi_c:[0,+\infty)\to[0,+\infty)$ is increasing.
If $\bar\gamma$ is a~Puiseux series and 
${\rm cont}(\Zer f^{[c]},\bar\gamma)=h$
then $\ord f^{[c]}(\bar\gamma)=\phi_c(h)$.
\end{Lemma}

\begin{proof}
By equality~(\ref{Eq1}) we get
$\overline{q_{i+1}}=\overline{q_i} + e_{i}(\overline{h_{i+1}}-\overline{h_{i}})$
for $i=0,\dots, s-1$,
hence the numbers $\overline{q_i}$ form an increasing sequence.

Let $h={\rm cont}(\Zer f^{[c]},\overline{\gamma} )={\rm ord}(\overline{\gamma}-\overline{\alpha})$ for some $\alpha\in \Zer f$. 
Assume that $h\in(\overline{h_r},\overline{h_{r+1}}]$.
Then, by the strong triangle inequality and  Lemma~\ref{contact}, we get
\begin{eqnarray*}
\ord f^{[c]}(\overline{\gamma})  &=& \sum_{j=1}^n \ord(\overline{\gamma}-\overline{\alpha_j}) \\
&=& \sum_{\ord\,(\overline{\alpha_j}-\overline{\alpha})\leq h_r}
\ord\,(\overline{\gamma}-\overline{\alpha_j}) + 
\sum_{\ord\,(\overline{\alpha_j}-\overline{\alpha})>h_r}
\ord\,(\overline{\gamma}-\overline{\alpha_j}) \\
&=& \sum_{\ord\,(\overline{\alpha_j}-\overline{\alpha})\leq h_r}
\ord\,(\overline{\alpha}-\overline{\alpha_j}) + 
\sum_{\ord\,(\overline{\alpha_j}-\overline{\alpha})>h_r}
\ord\,(\overline{\gamma}-\overline{\alpha}) \\
&=& \sum_{i=1}^{r} (e_{i-1}-e_i)\overline{h_i} +  e_{r}h= \overline{q_{r}}+e_r(h-\overline{h_r}).
\end{eqnarray*}
\end{proof}

Let $\alpha \in \Zer f$ and $h_r$ be a characteristic exponent of $f$.
By definition, the $h_r$--{\it truncation} of $\alpha$
 is the fractional power series $\trunc_r(  \alpha)$ 
 obtained from $\alpha$ by omitting all terms of order $\geq h_r$. 
 We denote by $f_r$ the minimal polynomial of $\trunc_r(\alpha)$ 
 over the field $K$. 
 As we will see in the lemma below, this polynomial does not depend on $\alpha$.

\begin{Lemma} $\quad$
\begin{itemize}
\item[{\rm (i)}] ${\rm Zer}\, f_r=\{\,\trunc_r(\alpha_j):j=1,\dots,n\,\}$;
\item[{\rm (ii)}] $f_r \in \K[[X]]$ is monic, irreducible and  quasi-ordinary;
\item[{\rm (iii)}] $\deg f_r=n_1\cdots n_{r-1}$; 
\item[{\rm (iv)}] ${\rm Char}(f_r)=\{h_1,\dots,h_{r-1}\}$.
\end{itemize}
\end{Lemma}
\begin{proof}
Since $\trunc_r(\alpha)\in \K[[X^{1/n}]]$ and  $L$ is normal over $K$, 
all the roots of the~polynomial $ f_{r}$ are elements of $L$.  
It is easy to see that $\sigma(\trunc_r(\alpha))=\trunc_r(\sigma(\alpha))$ for every $\sigma\in {\rm Gal}(L/K)$ . 
The polynomial $f$ is irreducible over the field $K$, 
so ${\rm Gal}(L/K)$ acts transitively on the set $\Zer f$ 
and hence on the set $\Zer f_r$, as well. 
This implies~(i) and (ii).

If $d(\alpha_i,\alpha_j)\leq h_{r-1}$, 
then $d(\alpha_i,\alpha_j)= d({\rm trunc}_r(\alpha_i),{\rm trunc}_r(\alpha_j))$ and if $d(\alpha_i,\alpha_j)\linebreak \geq h_r$, then    ${\rm trunc}_r(\alpha_i)= {\rm trunc}_r(\alpha_j)$. 
Thus (iv) holds true and, as a consequence, we also obtain (iii).
\end{proof}

\section{Proof of Theorem \ref{Th:irred1}}

The proof  will be organized as a sequence of claims. 
We denote the roots of $f$ (respectively the roots of $f_{k+1}$) by 
$\alpha_1$, \dots, $\alpha_n$ (respectively by $\beta_1$, \dots, $\beta_l$, where $l=n_1\cdots n_k$). 
We will use the bar notation for polynomials and power series 
after the monomial substitution $X=T^c$. 

Let $c\in\Q_{+}^d$ be generic in the sense that the conclusion of Theorem~\ref{T21} 
for a~polynomial $g$ and any $\overline{\beta}$, $\overline{\beta'}\in \Zer \overline{f_{k+1}}$ is true. 

\medskip\noindent
\textbf{Claim 1.}
For every $\overline{\beta}\in {\rm Zer}\, \overline{f_{k+1}}$ 
there exists exactly one $\gamma\in {\rm Zer}\,\overline{g}$ such that 
$\ord(\gamma-\overline{\beta})> \bar h_k$.
\begin{proof}
By assumptions of the theorem we get $\ord \Res(\overline{f}, \overline{g})> (\deg g)\bar q_k$.
If $\cont(\Zer\overline{f},\gamma) \leq \bar h_k$ for all the roots $\gamma$ of $\overline{g}$, 
then by Lemma~\ref{L:ineq} we obtain
$\ord \Res(\overline{f},\overline{g}) = 
\sum_{\gamma\in\Zer 
\overline{g}} \ord \overline{f}(\gamma) 
\leq (\deg \overline{g})\bar q_k$. 
It follows that $\ord(\overline{\alpha}-\gamma)> \overline{h_k}$ 
for some $\gamma\in \Zer\, \overline{g}$ and $\alpha\in\Zer f$.

Let $\beta=\trunc_{k+1}(\alpha)$.
Since $\ord (\overline{\beta}-\overline{\alpha})>\overline{h_k}$, we get 
$\ord(\overline{\beta}-\gamma)> \overline{h_k}$ and 
consequently $\cont(\Zer\,\overline{g},\overline{\beta}) > \overline{h_k}$.
It follows from Theorem~\ref{T21} that for every 
$\overline{\beta'}\in \Zer \overline{f_{k+1}}$ 
there exists $\gamma\in {\rm Zer}\,\overline{g}$ such that 
$\ord(\gamma-\overline{\beta'})>\overline{h_k}$.

Take any $\overline{\beta},\overline{\beta'}\in\Zer\overline{f_{k+1}}$ and 
$\gamma,\gamma'\in {\rm Zer}\,\overline{g}$ such that 
$\ord(\gamma-\overline{\beta})>\overline{h_k}$ and
$\ord(\gamma'-\overline{\beta'})>\overline{h_k}$. Assume that 
$\overline{\beta}\neq \overline{\beta'}$. Then $\gamma\neq\gamma'$.

Indeed, if $\gamma=\gamma'$ then 
$\ord(\overline{\beta}-\overline{\beta'})>\overline{h_k}$ and we arrive at contradiction. 
From the above and the assumption $\deg g\leq n_1\cdots n_k=\deg f_{k+1}$ 
we obtain that $\overline{g}$ has exactly $n_1\cdots n_k$ roots, which  completes the  proof.
\end{proof}

Using Claim 1 we may assume, without loss of generality, that 
$\Zer\,\overline{g}=\{\gamma_1,\dots,\gamma_l\}$, 
where $\ord(\gamma_i-\overline{\beta_i})> \overline{h_k}$ for all $1\leq i\leq l$. 
It follows immediately from the strong triangle inequality that 
\begin{equation}\label{orders}
 \ord(\gamma_i - \gamma_j)=\ord(\overline{\beta_i}-\overline{\beta_j})
\quad\mbox{for all $1\leq i<j\leq l$}.
\end{equation}
 Hence the orders of the discriminants of polynomials $\overline{g}$ and $\overline{f_{k+1}}$ are equal.  Therefore, by Lemma~\ref{cor:delta}, the Newton polytopes 
of the discriminants of $g$ and $f_{k+1}$ are equal too, so we conclude that $g$ is quasi-ordinary.
%We denote the roots of $g$ by $\gamma_1$,\dots, $\gamma_m$. 

\medskip\noindent
\textbf{Claim 2.}
Let $v$ be a vertex of $\Delta(\gamma - \gamma')$ for $\gamma, \gamma' \in\Zer\, g$. 
Then $v\in\{h_1,\dots,h_k\}$.
\begin{proof}
Since $g$ is quasi-ordinary, $\Delta(\gamma-\gamma')$ has only one vertex. Thus
for every $c\in\Q_{+}^d$ we have $\ord(\gamma^{[c]}-\gamma'^{[c]}))=\langle c,v\rangle$.
It follows from~(\ref{orders}) that $\langle c,v\rangle\in \{\langle c,h_i\rangle: 1\leq i\leq k\}$.
Observe that if $v\ne h_i$, then the set
$\{c'\in\Q_{+}^d: \langle c',v\rangle=\langle c',h_i\rangle \}$ 
is contained in a finite union of hyperplanes, hence
is nowhere dense. 
This implies that $v=h_i$ for some $i\in\{1,\dots,k\}$, since $c$ is generic. 
%We conclude that ${\rm Char}(g)\subset \{h_1,\dots,h_k\}$.
\end{proof}

\noindent \textbf{Claim 3.}
Let $v$ be a vertex of $\Delta(\gamma-\beta)$ for 
$\gamma\in \Zer\,g$ and  $\beta\in\Zer f_{k+1}$. 
Then $v\in\{h_1,\dots,h_k\}$ or $v>h_k$. 
\begin{proof}
By the strong triangle inequality and~(\ref{orders}) we obtain 
$\ord(\gamma^{[c]}-\beta^{[c]}) \in \{\langle c,h_i\rangle: 1\leq i\leq k\}$ or 
$\ord(\gamma^{[c]}-\beta^{[c]})>\langle c,h_k\rangle$.

Let $v$ be a vertex of $\Delta(\gamma-\beta)$ which is not in $\Delta(X^{h_k})$. 
Then there exists an open set $U\subset \Q_{+}^d$ such that a face 
$\Delta(\gamma-\beta)^c=\{v\}$ 
and $\langle c,v\rangle < \langle c,h_k\rangle$ for all $c\in U$.
Hence  we have 
$\ord(\gamma^{[c]}-\gamma'^{[c]}))=\langle c,v\rangle<\langle c,h_k\rangle$.
Using the same argument as in the proof of Claim~2, we conclude that 
$v\in \{h_1,\dots, h_{k}\}$ which completes the proof. 
\end{proof}

\noindent
\textbf{Claim 4.}
For every $\beta\in \Zer f_{k+1}$ there exists exactly one $\gamma\in \Zer\,g$ such that 
$\Delta(\gamma-\beta)\subsetneq \Delta(X^{h_k})$. 
\begin{proof}
Let $\Delta=\Delta(\gamma-\beta)$ for $\gamma\in \Zer\,g$ and  $\beta\in\Zer f_{k+1}$.
Then by Claim 3 two cases are possible:
either some $h_i\in \{h_1,\dots, h_k\}$ is a vertex of  $\Delta$ 
and $\ord(\overline{\gamma}-\overline{\beta})=\overline{h_i}$
or 
$\Delta\subsetneq \Delta(X^{h_k})$ 
and $\ord(\overline{\gamma}-\overline{\beta})>\overline{h_k}$.
To finish the proof it is enough to use Claim~1. 
\end{proof}

We will show that ${\rm Gal}(L/K)$
acts transitively on the set $\Zer\,g$. 
Indeed, take arbitrary $\gamma, \gamma'\in \Zer\,g$. By Claim~4 there exist unique
$\beta,\beta'\in \Zer f_{k+1}$ such that 
$\Delta(\gamma-\beta)\subsetneq \Delta(X^{h_k})$ and 
$\Delta(\gamma'-\beta')\subsetneq \Delta(X^{h_k})$. 
Take $\sigma\in {\rm Gal}(L/K)$ such that $\sigma(\beta)=\beta'$. 
Then by Remark~\ref{Galois} we have $\Delta(\sigma(\gamma)-\beta')\subsetneq \Delta(X^{h_k})$, 
hence $\sigma(\gamma)=\gamma'$. 

It follows from the above that the polynomial $g$ is irreducible. 
From~(\ref{orders}) and Claim~2 we deduce that $(h_1,\dots, h_k)$ is the characteristic of $g$. 
Point~(ii) of the theorem follows directly from Claim~4. 
 
Now we prove statements (iii) and (iv) of Theorem~\ref{Th:irred1}. 
Assume that  $X^{(\deg g)q_{k+1}}$ divides  $\Res(f,g)$. 
If the monomial $X^{(\deg g)q_{k+1}}$ does not appear in $\Res(f,g)$ then 
by the first part of the theorem we obtain $\deg g = n_1\cdots n_{k+1}$,
which contradicts the assumption $\deg g \leq n_1\cdots n_k$. Thus 
$\Res(f,g) = u(X)\, X^{(\deg g)q_{k+1}}$, where $u(0)\neq0$.

By Claim~4  for every $\gamma\in \Zer\,g$ there exists $\beta\in \Zer f_{k+1}$ 
such that $\Delta(\gamma-\beta)\subsetneq \Delta(X^{h_k})$. 
Suppose that the Newton polytope $\Delta=\Delta(\gamma-\beta)$ 
has a vertex which is not contained in $\Delta(X^{h_{k+1}})$. 
Then there exists $c\in \Q_{+}^d$ such that $\Delta^c=\{v\}$ and 
$\langle c,v\rangle<\langle c,h_{k+1}\rangle$. 
Thus $\cont(\Zer \overline{f},\overline{\gamma})<\overline{h_{k+1}}$. 
Since for any $\sigma\in {\rm Gal}(L/K)$ we have 
$\Delta(\gamma-\beta)=\Delta(\sigma(\gamma)-\sigma(\beta))$, the same is true for 
any $\gamma'\in \Zer\, g$. 
Then by Lemma~\ref{L:ineq} we get 
$\ord(\Res(\overline{f},\overline{g}))<(\deg g)\overline{q_{k+1}}$ 
and we arrive at contradiction. 

We proved that $\Delta(\gamma-\beta)\subseteq \Delta(X^{h_{k+1}}) $ which gives (iv) 
of Theorem~\ref{Th:irred1}.
%%%%%%%%%%%%%%%%%%%%%%%%%%%%%%%%%%%%%%%%

\section{Logarithmic distance}
For any irreducible Weierstrass polynomials $f$, $g\in \K[[X]][Y]$ we define the Newton polytope  
$\cont_A(f,g)=\frac{1}{(\deg f)(\deg g)}\Delta(\Res(f,g))$ 
called the {\it logarithmic distance} between $f$ and $g$.
We introduce the partial order in the set of Newton polytopes: 
$\Delta_1\geq \Delta_2$ if and only if $\Delta_1\subset \Delta_2$. 

For any  irreducible quasi-ordinary Weierstrass polynomials $f$, $g$, $h$ the strong triangle inequality
$\cont_A(f,g)\geq \inf\{\cont_A(f,h),\cont_A(h,g)\}$ holds true, where $\inf\{A,B\}$ 
denotes the Newton polytope spanned by the union of $A$ and $B$. Let us prove it now.

For $\alpha =\sum_{a\in\Q^d} c_a X^a\in\K[[X^{1/\mathbb{N}}]]$ and $\omega\in\mathbb{R}_{+}^d$ we define a {\it weighted order} ${\rm ord}_w(\alpha):=\min\{\langle \omega,a\rangle:c_a\ne 0\}$ and a {\it weighted contact} between quasi-ordinary polynomials $f,g\in\K[[X]][Y]$ as follows: $${\rm cont}_{\omega}(f,g):=\frac{1}{\deg f \deg g}\sum_{\substack{\alpha\in {\rm Zer}\, f\\ \beta\in {\rm Zer}\, g}} {\rm ord}_{\omega} (\alpha-\beta)=\frac{1}{\deg f \deg g} l(\omega,\Delta(\Res(f,g))),$$
where $l(\omega,\Delta(\Res(f,g))):=\min\{\langle \omega,a\rangle:a\in\Delta(\Res(f,g)\}$.
 
 For every irreducible quasi-ordinary polynomials $f,g\in \mathbb{K}[[X]][Y]$ and for any $\gamma,\gamma'\in {\rm Zer}\, g$ we have ${\rm ord}_{\omega} f(\gamma)={\rm ord}_{\omega} f(\gamma')$. Therefore, using the same method as in the proof of \cite[Proposition 2.2]{Pl}, we get 
for any irreducible quasi-ordinary polynomials $f$, $g$,  $h \in \mathbb{K}[[X]][Y]$ a strong triangle inequality ${\rm cont}_{\omega}(f,g)\geq \min\{{\rm cont}_{\omega}(f,h),{\rm cont}_{\omega}(h,g)\}$.

\begin{Lemma}
Assume that $\Delta_1,\Delta_2,\Delta_3$ are  Newton polytopes. If \begin{equation}l(\omega,\Delta_1)\geq \min\{l(\omega,\Delta_2), l(\omega,\Delta_3\}\label{e1}\end{equation} for all $\omega \in \R_{+}^d$, then $\Delta_1\geq \inf\{\Delta_2,\Delta_3\}$. 
\end{Lemma}
\begin{proof}
Suppose that the inequality $\Delta_1\geq \inf\{\Delta_1,\Delta_2\}$ is false. Therefore there exists $v\in \Delta_1\setminus {\rm conv}(\Delta_2\cup\Delta_3)$ and then we can find a~linear form $L$ such that $L(v)<L(x)$ for all $x\in {\rm conv}(\Delta_2\cup \Delta_3)$. It means that $\langle \omega , v\rangle < \langle \omega, x\rangle $, $x\in {\rm conv}(\Delta_2\cup \Delta_3)$, for some $\omega \in \mathbb{R}_{+}^d$. Thus $l(\omega,\Delta_1)<l(\omega,{\rm conv}(\Delta_1\cup\Delta_2))$ and the inequality (\ref{e1}) does not hold.
\end{proof}
The above lemma implies immediately the strong triangle inequality for the~logarithmic distance of irreducible quasi-ordinary Weierstass polynomials.

Unfortunately, the strong triangle inequality does not extend to a wider class 
of irreducible Weierstrass polynomials as the following examples show. 

\begin{Example}
{\rm Let $f=Y$, $g=Y-X_1-X_2^2$, $h=Y^2-(X_1+X_2)Y+2X_1^3+X_2^3$. 
The polynomials $f$, $g$, $h$ are irreducible in $\K[[X_1,X_2]][Y]$. 
We have
$\Res(f,g)=-X_1-X_2^2$,
$\Res(f,h)=2X_1^3+X_2^3$,
$\Res(g,h)=X_1X_2^2-X_1X_2+2X_1^3+X_2^4$, hence 
there is no strong triangle inequality between 
${\rm cont}_A(f,g)$, ${\rm cont}_A(f,h)$ and ${\rm cont}_A(h,g)$ 
as illustrated in the following picture. }
\end{Example}

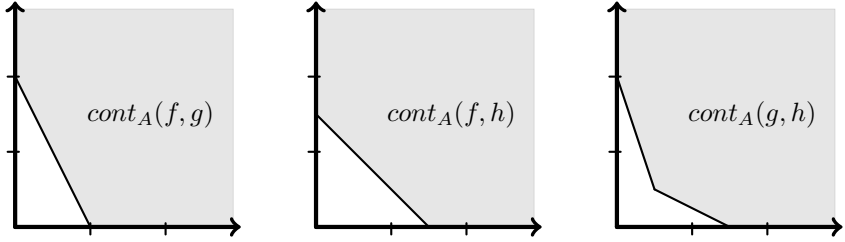
\begin{tikzpicture}[
     scale = 1,
     foreground/.style = { ultra thick },
     background/.style = { dashed },
     line join=round, line cap=round
   ]
   
   \draw[fill=black, opacity=0.1] (1,0)--(2.9,0)--(2.9,2.9)--(0,2.9)--(0,2)--cycle;
   \draw[foreground,->] (0,0)--+(3,0);
   \draw[foreground,->] (0,0)--+(0,3);
   \draw[thick] (1,-0.1)--+(0,0.15);
   \draw[thick] (2,-0.1)--+(0,0.15);
   \draw[thick] (-0.1,1)--+(0.15,0);
   \draw[thick] (-0.1,2)--+(0.15,0);
   \draw[thick] (1,0)--(0,2);
   \draw (1,-0.15) node[below] {$ $};
   \draw (-0.15,1) node[left] {$ $};
   \foreach \x in{} {
    \foreach \y in{}{
     \draw[fill, opacity=0.9]  (\x,\y) circle (0.5pt);
     }
    }
   \node at (1.8,1.5) {$cont_A(f,g)$}; 
  
   [
       scale = 1,
     foreground/.style = { ultra thick },
     background/.style = { dashed },
     line join=round, line cap=round
   ]
   
   \draw[fill=black, opacity=0.1] (5.5,0)--(6.9,0)--(6.9,2.9)--(4,2.9)--(4,1.5)--cycle;
   \draw[foreground,->] (4,0)--+(3,0);
   \draw[foreground,->] (4,0)--+(0,3);
   \draw[thick] (5,-0.1)--+(0,0.15);
   \draw[thick] (6,-0.1)--+(0,0.15);
   \draw[thick] (3.9,1)--+(0.15,0);
   \draw[thick] (3.9,2)--+(0.15,0);
   \draw[thick] (5.5,0)--(4,1.5);
   \draw (1,-0.15) node[below] {$ $};
   \draw (-0.15,1) node[left] {$ $};
   \foreach \x in{} {
    \foreach \y in{}{
     \draw[fill, opacity=0.9]  (\x,\y) circle (0.5pt);
     }
    }
   \node at (5.8,1.5) {$cont_A(f,h)$}; 
   
    [
       scale = 1,
     foreground/.style = { ultra thick },
     background/.style = { dashed },
     line join=round, line cap=round
   ]
   
   \draw[fill=black, opacity=0.1] (9.5,0)--(10.9,0)--(10.9,2.9)--(8,2.9)--(8,2)--(8.5,0.5)--cycle;
   \draw[foreground,->] (8,0)--+(3,0);
   \draw[foreground,->] (8,0)--+(0,3);
   \draw[thick] (9,-0.1)--+(0,0.15);
   \draw[thick] (10,-0.1)--+(0,0.15);
   \draw[thick] (7.9,1)--+(0.15,0);
   \draw[thick] (7.9,2)--+(0.15,0);
   \draw[thick] (9.5,0)--(8.5,0.5);
   \draw[thick] (8.5,0.5)--(8,2);
   \draw (1,-0.15) node[below] {$ $};
   \draw (-0.15,1) node[left] {$ $};
   \foreach \x in{} {
    \foreach \y in{}{
     \draw[fill, opacity=0.9]  (\x,\y) circle (0.5pt);
     }
    }
   \node at (9.8,1.5) {$cont_A(g,h)$};
  \end{tikzpicture}

\begin{Example}
{\rm Let $f=Y-2X_1^3$, $g=(Y-X_1)(Y-X_1^3-X_1^4)+X_2$, $h=Y-X_1^3$. 
Then $\Res(f,g)=-2X_1^7+2X_1^6+X_1^5-X_1^4+X_2$, 
$\Res(f,h)=X_1^3$, 
$\Res(g,h)=-X_1^7+X_1^5+X_2$.
The polynomials $f$, $g$, $h$ are irreducible, $fh$ is quasi-ordinary 
and the inequality ${\rm cont}_A(f,g)\geq \inf\{{\rm cont}_A(f,h),{\rm cont}_A(h,g)\}$
does not hold (see the picture below).}
\end{Example}

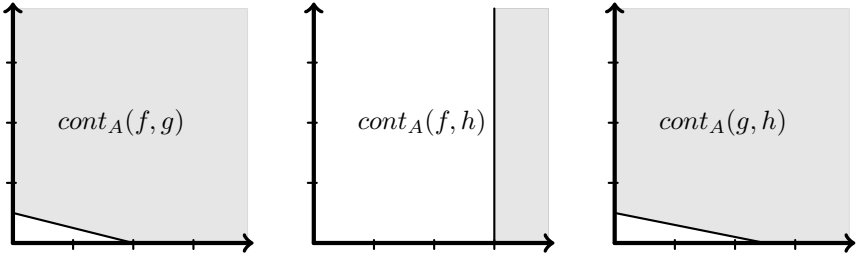
\begin{tikzpicture}[
     scale = 0.8,
     foreground/.style = { ultra thick },
     background/.style = { dashed },
     line join=round, line cap=round
   ]
   
   \draw[fill=black, opacity=0.1] (2,0)--(3.9,0)--(3.9,3.9)--(0,3.9)--(0,0.5)--cycle;
   \draw[foreground,->] (0,0)--+(4,0);
   \draw[foreground,->] (0,0)--+(0,4);
   \draw[thick] (1,-0.1)--+(0,0.15);
   \draw[thick] (2,-0.1)--+(0,0.15);
   \draw[thick] (3,-0.1)--+(0,0.15);
   \draw[thick] (-0.1,1)--+(0.15,0);
   \draw[thick] (-0.1,2)--+(0.15,0);
   \draw[thick] (-0.1,3)--+(0.15,0);
   \draw[thick] (2,0)--(0,0.5);
   \draw (1,-0.15) node[below] {$ $};
   \draw (-0.15,1) node[left] {$ $};
   \foreach \x in{} {
    \foreach \y in{}{
     \draw[fill, opacity=0.9]  (\x,\y) circle (0.5pt);
     }
    }
   \node at (1.8,2) {$cont_A(f,g)$}; 
  
   [
       scale = 0.8,
     foreground/.style = { ultra thick },
     background/.style = { dashed },
     line join=round, line cap=round
   ]
   
   \draw[fill=black, opacity=0.1] (8,0)--(8.9,0)--(8.9,3.9)--(8,3.9)--cycle;
   \draw[foreground,->] (5,0)--+(4,0);
   \draw[foreground,->] (5,0)--+(0,4);
   \draw[thick] (6,-0.1)--+(0,0.15);
   \draw[thick] (7,-0.1)--+(0,0.15);
   \draw[thick] (8,-0.1)--+(0,0.15);
   \draw[thick] (4.9,1)--+(0.15,0);
   \draw[thick] (4.9,2)--+(0.15,0);
   \draw[thick] (4.9,3)--+(0.15,0);
   \draw[thick] (8,0)--(8,3.9);
   \draw (1,-0.15) node[below] {$ $};
   \draw (-0.15,1) node[left] {$ $};
   \foreach \x in{} {
    \foreach \y in{}{
     \draw[fill, opacity=0.9]  (\x,\y) circle (0.5pt);
     }
    }
   \node at (6.8,2) {$cont_A(f,h)$}; 
   
    [
       scale = 0.8,
     foreground/.style = { ultra thick },
     background/.style = { dashed },
     line join=round, line cap=round
   ]
   
   \draw[fill=black, opacity=0.1] (12.5,0)--(13.9,0)--(13.9,3.9)--(10,3.9)--(10,0.5)--cycle;
   \draw[foreground,->] (10,0)--+(4,0);
   \draw[foreground,->] (10,0)--+(0,4);
   \draw[thick] (11,-0.1)--+(0,0.15);
   \draw[thick] (12,-0.1)--+(0,0.15);
    \draw[thick] (13,-0.1)--+(0,0.15);
   \draw[thick] (9.9,1)--+(0.15,0);
   \draw[thick] (9.9,2)--+(0.15,0);
   \draw[thick] (9.9,3)--+(0.15,0);
   \draw[thick] (12.5,0)--(10,0.5);
   \draw (3,-0.15) node[below] {$ $};
   \draw (-0.15,1) node[left] {$ $};
   \foreach \x in{} {
    \foreach \y in{}{
     \draw[fill, opacity=0.9]  (\x,\y) circle (0.5pt);
     }
    }
   \node at (11.8,2) {$cont_A(g,h)$};
  \end{tikzpicture}

The results of the first section can 
be reformulated in terms of the logarithmic distance.

\begin{Theorem}\label{Th:irred3}
Let $f\in \K[[X]][Y]$ be a quasi-ordinary irreducible Weierstrass polynomial 
of characteristic $(h_1,\dots,h_s)$
and let $g\in \K[[X]][Y]$ be a Weierstrass polynomial such that 
$\deg g \leq \deg f_{k+1}$ and 
$\cont_A(f,g) > \cont_A(f,f_k)$.
Then $g$ is an irreducible quasi-ordinary polynomial of characteristic $(h_1,\dots, h_{k})$
and $\deg g = \deg f_{k+1}$.
Moreover, if $\cont_A(f,g)\geq \cont_A(f,f_{k+1})$ then $\cont_A(f,g) =\cont_A(f,f_{k+1})$.
\end{Theorem}

\section{The Abhyankar-Moh irreducibility criterion}

 In this section we will show that our main result is a generalization of the~well-known Abhyankar-Moh irreducibility criterion (see e.g. \cite[Theorem~1.2]{am}).

At the beginning let us recall the classical Weierstrass preparation theorem for the ring $\mathbb{K}[[X,Y]]$.

\begin{Theorem}[Weierstrass]
Assume that $f=\sum_{i=0}^{\infty}a_i Y^i\in \mathbb{K}[[X,Y]]$ and there exists $m>0$ such that $a_i(0)= 0$ for $i<m$ and $a_m(0)\ne0$. Then there exist unique $u_f,f_1\in\mathbb{K}[[X,Y]]$ such that $f=u_f f_1$, $u_f(0)\ne 0$ and $f_1$ is a~Weierstrass polynomial with respect to the variable $Y$. \label{We}
\end{Theorem}

The polynomial $f_1$ from the above theorem is called the {\it Weierstrass polynomial} of $f$. If we assume additionally that $f\in \mathbb{K}[[X,Y]]$ is irreducible, then its Weierstrass polynomial is  irreducible in the ring $\mathbb{K}[[X]][Y]$ and  the  characteristic of this polynomial will be denoted by ${\rm Char}(f)$.

For $f,g\in \K[[X,Y]]$ we define the {\it intersection multiplicity number}  $i_0(f,g)$ as the~dimension of the $\K$--vector space $\K[[X,Y]]/\langle f,g\rangle$.

\begin{Theorem}[Abhyankar-Moh] Let $f$, $g\in \K[[X,Y]]$.
Assume that $f$ is irreducible, $i_0(f,X)=n<+\infty$ and ${\rm Char}(f)=\{h_1,\dots,h_s\}$. If $i_0(g,X)=n$ and $i_0(f,g)>nq_s$, then $g$ is irreducible and ${\rm Char}(g)={\rm Char}(f)$.
\end{Theorem}

\begin{proof}
Let $f_1$ and  $g_1$ be the Weierstass polynomials of $f$ and $g$. Then 
$\deg f_1=i_0(f,X)=n$, $\deg g_1=i_0(g,X)=n$ and ${\rm ord}\,\Res(f_1, g_1)=i_0(f,g)>nq_s$.
Since $f$ is irreducible, the polynomial $f_1\in\mathbb{K}[[X]][Y]$ is also irreducible.
Therefore Theorem \ref{Th:irred1} with $k=s$ implies 
that $g_1$ is irreducible of characteristic $(h_1,\dots,h_s)$ and the~theorem follows.
\end{proof}
%%%%%%%%%%%%%%%%%%%%%%%%%%%%%%%%%%%%%

\medskip
\noindent
{\small   Janusz Gwo\'zdziewicz\\
Institute of Mathematics\\
Pedagogical University of Cracow\\
Podchor\c a{\accent95 z}ych 2\\
PL-30-084 Cracow, Poland\\
e-mail: gwozdziewicz@up.krakow.pl}

\medskip
\noindent
{\small   Beata Hejmej\\
Institute of Mathematics\\
Pedagogical University of Cracow\\
Podchor\c a{\accent95 z}ych 2\\
PL-30-084 Cracow, Poland\\
e-mail: bhejmej1f@gmail.com}

\end{document}